\documentclass[10pt]{amsart}

\usepackage{amsmath,amsfonts, latexsym,graphicx}

\newcommand{\C}{{\mathbb C}}
\newcommand{\Z}{{\mathbb Z}}

\newcommand{\Q}{{\mathbb Q}}

\newcommand{\W}{{\mathcal W}}

\newcommand{\hyp}{{\mathbb H}}

\newcommand{\arrow}[1]{\stackrel{#1}{\longrightarrow}}
\newcommand{\Adot}{\mathbf A^\bullet}

\newcommand{\ob}{{\stackrel{\circ}B}}

\newcommand{\vdual}{{\mathcal D}}

\newtheorem{defn0}{Definition}[section]
\newtheorem{prop0}[defn0]{Proposition}
\newtheorem{conj0}[defn0]{Conjecture}
\newtheorem{thm0}[defn0]{Theorem}
\newtheorem{lem0}[defn0]{Lemma}
\newtheorem{corollary0}[defn0]{Corollary}
\newtheorem{example0}[defn0]{Example}
\newtheorem{remark0}[defn0]{Remark}
\newtheorem{question0}[defn0]{Question}

\newenvironment{thm}{\begin{thm0}}{\end{thm0}}
\newenvironment{lem}{\begin{lem0}}{\end{lem0}}
\newenvironment{cor}{\begin{corollary0}}{\end{corollary0}}

\newcommand{\lemref}[1]{Lemma~\ref{#1}}

\title{Natural Commuting of Vanishing Cycles and the Verdier Dual}



\author{David B. Massey}

\date{}

\begin{document}

\baselineskip = 14pt

\begin{abstract} We prove that the shifted vanishing cycles and nearby cycles commute with Verdier dualizing up to a {\bf natural} isomorphism, even when the coefficients are not in a field.
\end{abstract}

\maketitle

\sloppy




\section{Introduction} In this short, technical, paper, we prove a result whose full statement is missing from the literature, and which may be surprising to some experts in the field.  To state this result, we need to use technical notions and notations; references are  \cite{kashsch}, \cite{dimcasheaves}, \cite{schurbook}, and Appendix B of \cite{numcontrol}. We should remark immediately that the definition that we use (see below) for the vanishing cycles is the standard one, which is shifted by one from the definition in   \cite{kashsch}

We fix a base ring, $R$, which is a commutative, regular, Noetherian ring, with finite Krull dimension (e.g., $\Z$, $\Q$, or $\C$). Throughout this paper, by a topological space, we will mean a locally compact space. When we write that $\Adot$ is complex of sheaves on a topological space, $X$, we mean that $\Adot$ is an object in $D^b(X)$, the derived category of bounded, complexes of sheaves of $R$-modules on $X$. When $X$ is complex analytic, we may also require that $\Adot$ is (complex) constructible, and write $\Adot\in D^b_c(X)$. We let $\vdual=\vdual_X$ denote the Verdier dualizing operator on $D^b_c(X)$. We will always write simply $\vdual$, since the relevant topological space will always be clear.

Suppose that $f:X\rightarrow\C$ is a complex analytic function, where $X$ is an arbitrary complex analytic space, and suppose that we have a complex of sheaves $\Adot$ on $X$. We let $\psi_f$ and $\phi_f$ denote the nearby and vanishing cycle functors, respectively. Henceforth, we shall always write these functors composed with a shift by $-1$, that is, we shall write $\psi_f[-1]:=\psi_f\circ [-1]$ and $\phi_f[-1]:=\phi_f\circ [-1]$. In order to eliminate any possible confusion over indexing/shifting: with the definitions that we are using, if $F_{f, p}$ denotes the Milnor fiber of $f$ at $p\in f^{-1}(0)$ inside the open ball $\ob_\epsilon(p)$ (using a local embedding into affine space), then $H^k(\psi_f[-1]\Adot)_p\cong \hyp^{k-1}(F_{f, p}; \Adot)$ and $H^k(\phi_f[-1]\Adot)_p\cong \hyp^{k}(\ob_\epsilon(p), F_{f, p}; \Adot)$.

\medskip

\noindent The questions in which we are interested is: 

\smallskip

{\bf Do there exist isomorphisms $\vdual\circ\psi_f[-1]\cong \psi_f[-1]\circ\vdual$ and $\vdual\circ\phi_f[-1]\cong \phi_f[-1]\circ\vdual$, even if $R$ is {\bf not} a field, and do there exist such isomorphisms which are {\bf natural}?}

\medskip

The answer is: {\bf yes}. In this paper, we show quickly that such natural isomorphisms exist, even if the base ring is not a field. 

\medskip

Is this result known and/or surprising? Some references, such as \cite{brymont}, \cite{dimcasheaves}, and our own Appendix B of \cite{numcontrol},   state that there exist non-natural isomorphisms, and require that the base ring is a field.  In \cite{schurbook},  Corollary 5.4.4, Sch\"urmann proves the natural isomorphism exists on the stalk level, even when $R$ is not a field.

In the $l$-adic algebraic context, Illusie proves in \cite{illusie} that the Verdier dual and nearby cycles commute, up to natural isomorphism. M. Saito proves the analogous result in the complex analytic setting, with field coefficients, in \cite{msaitomod} and \cite{msaitodual}. One can obtain our full result by combining Proposition 8.4.13, Proposition 8.6.3, and Exercise VIII.15 of \cite{kashsch}, though our proof here is completely different. In fact, our proof is similar to the discussion on duality of local Morse data following Remark 5.1.7 in \cite{schurbook}.

\smallskip

Our proof is relatively simple, and consists of three main steps: proving a small lemma about pairs of closed sets which cover a space, using a convenient characterization/definition of the vanishing cycles, and using that the stratified critical values of $f$ are locally isolated. The nearby cycle result follows as a quick corollary of the result for vanishing cycles.

\smallskip

We thank J\"org Sch\"urmann for providing valuable comments and references for related results.

\section{Two Lemmas}

We shall use $\simeq$ to denote natural isomorphisms of functors.

\smallskip

The following is an easy generalization of the fact that, if $i$ is the inclusion of an open set, then $i^*\simeq i^!$ (see, for instance, Corollary 3.2.12 of \cite{dimcasheaves}.

\begin{lem}\label{lem1} Suppose $Z$ is a closed subset of $X$, and let $j:Z\hookrightarrow X$ denote the inclusion. Let $D^b_Z(X)$ denote the full subcategory of $D^b(X)$ of complexes $\Adot$ such that $Z\cap\operatorname{supp}\Adot$ is an open subset of  $\operatorname{supp}\Adot$. Then, there is a natural isomorphism of functors $j^!\simeq j^*$ from $D^b_Z(X)$ to $D^b(Z)$.
\end{lem}
\begin{proof} Let $\Adot\in D^b_Z(X)$. We will show that the natural map $j^!\rightarrow j^*$ of functors  from $D^b(X)$ to $D^b(Z)$ yields an isomorphism $j^!\Adot \rightarrow j^*\Adot$.

Let $Y:=\operatorname{supp}\Adot$. Let $m:Y\hookrightarrow X$, $\hat m:Z\cap Y\hookrightarrow Z$, and $\hat\jmath:Z\cap Y\hookrightarrow Y$ denote the inclusions. Then,
$$
j^!\Adot \ \cong \ j^!m_*m^*\Adot \ \cong \ \hat m_*\hat\jmath^!m^*\Adot \ \cong \ \hat m_!\hat\jmath^*m^*\Adot \ \cong \ j^*m_!m^*\Adot \ \cong \ j^*\Adot,
$$
where we used, in order, that $\Adot\cong m_*m^*\Adot$, since $Y$ is the support of $\Adot$, Proposition 2.6.7 of \cite{kashsch} on Cartesian squares, that $\hat m_*\simeq \hat m_!$ and $\hat\jmath^!\simeq \hat\jmath^*$, since $\hat m$ is a closed inclusion and $\hat\jmath$ is an open inclusion, Proposition 2.6.7 of \cite{kashsch} again, and that $\Adot\cong m_!m^*\Adot$.
\end{proof}

\bigskip

The lemma that we shall now prove certainly looks related to many propositions we have seen before, and may be known, but we cannot find a reference. The lemma tells us that, in our special case, the morphism of functors described in Proposition 3.1.9 (iii) of \cite{kashsch} is an isomorphism.

\begin{lem}\label{lem2} Let $X$ be a locally compact space, and let $Z_1$ and $Z_2$ be closed subsets of $X$ such that $X=Z_1\cup Z_2$. Denote the inclusion maps by: $j_1: Z_1\hookrightarrow X$,  $j_2: Z_2\hookrightarrow X$,  $\hat\jmath_1: Z_1\cap Z_2\hookrightarrow Z_2$, $\hat\jmath_2: Z_1\cap Z_2\hookrightarrow Z_1$, and $m=j_1\hat\jmath_2=j_2\hat\jmath_1:Z_1\cap Z_2\hookrightarrow X$.

Then, we have the following natural isomorphisms
\begin{equation}\label{eq}
\hat\jmath_1^*j_2^! \ \simeq \ m^*{j_2}_!j_2^! \ \simeq \ m^*{j_1}_*{\hat\jmath}_2{}_!\hat\jmath_2^!j_1^* \ \simeq \ \hat\jmath_2^!j_1^*.
\end{equation}
\end{lem}
\begin{proof} Let $i_1:X-Z_1\hookrightarrow X$ and $i_2:X-Z_2\hookrightarrow X$ denote the open inclusions. We make use of Proposition 2.6.7 of \cite{kashsch} on Cartesian squares repeatedly. We also use repeatedly that, if $j$ is a closed inclusion, then $j_*\simeq j_!$ and $j^*j_*\simeq j^*j_!\simeq \operatorname{id}$.

\medskip

We find
$$
m^*{j_2}_!j_2^! \ = \ (j_2\hat\jmath_1)^*{j_2}_!j_2^! \ \simeq \ \hat\jmath_1^*j_2^*{j_2}_!j_2^! \ \simeq \ \hat\jmath_1^*j_2^!, 
$$
which proves the first isomorphism in \eqref{eq}.

We also find
$$
m^*{j_1}_*{\hat\jmath}_2{}_!\hat\jmath_2^!j_1^* \ = \ (j_1\hat\jmath_2)^*{j_1}_*{\hat\jmath}_2{}_!\hat\jmath_2^!j_1^* \ \simeq \ \hat\jmath_2^*j_1^*{j_1}_*{\hat\jmath}_2{}_!\hat\jmath_2^!j_1^*\ \simeq \ \hat\jmath_2^*{\hat\jmath}_2{}_!\hat\jmath_2^!j_1^*\ \simeq \ \hat\jmath_2^!j_1^*,
$$
which proves the last isomorphism in \eqref{eq}.

\medskip

The middle isomorphism takes a bit more work.

\smallskip

As a first step, we find 
$$
{j_1}_*{\hat\jmath}_2{}_!\hat\jmath_2^!j_1^*  \ \simeq \ {j_1}_*{\hat\jmath}_2{}_*\hat\jmath_2^!j_1^* \ \simeq \ ({j_1\hat\jmath}_2)_*\hat\jmath_2^!j_1^* \ = \ ({j_2\hat\jmath}_1)_*\hat\jmath_2^!j_1^* \ \simeq \ {j_2}_*{\hat\jmath}_1{}_*\hat\jmath_2^!j_1^* \ \simeq
$$
$$
{j_2}_*j_2^!{j_1}_*j_1^* \ \simeq \ {j_2}_!j_2^!{j_1}_*j_1^*.
$$

\bigskip

Therefore, we need to show that
\begin{equation}\label{eq2}
m^*{j_2}_!j_2^! \ \simeq \ m^*{j_2}_!j_2^!{j_1}_*j_1^*.
\end{equation}

\bigskip

Consider the natural distinguished triangle
$$
{i_1}_!i_1^!\rightarrow\operatorname{id}\rightarrow {j_1}_*j_1^*\arrow{[1]}.
$$
Applying $m^*{j_2}_!j_2^!$ to this triangle, we obtain the natural distinguished triangle
$$
m^*{j_2}_!j_2^!{i_1}_!i_1^!\rightarrow m^*{j_2}_!j_2^!\rightarrow m^*{j_2}_!j_2^!{j_1}_*j_1^*\arrow{[1]}.
$$

We claim that $m^*{j_2}_!j_2^!{i_1}_!i_1^!=0$, which would yield \eqref{eq2} and finish the proof. We show this in two steps: we first show that ${j_2}_!j_2^!{i_1}_!i_1^!\cong {i_1}_!i_1^!$ and then that $m^*{i_1}_!i_1^!=0$.

\bigskip

Consider the natural distinguished triangle
$$
{j_2}_!j_2^!\rightarrow\operatorname{id}\rightarrow {i_2}_*i_2^*\arrow{[1]},
$$
applied to ${i_1}_!i_1^!$; this yields the natural distinguished triangle
$$
{j_2}_!j_2^!{i_1}_!i_1^!\rightarrow {i_1}_!i_1^!\rightarrow {i_2}_*i_2^*{i_1}_!i_1^!\arrow{[1]},
$$
However, $i_2^*{i_1}_! =0$, since $X=Z_1\cup Z_2$ and, thus, we conclude that the natural map ${j_2}_!j_2^!{i_1}_!i_1^!\rightarrow {i_1}_!i_1^!$ is an isomorphism.

It remains for us to show that $m^*{i_1}_!i_1^!=0$, but this is trivial, because 
$$
m^*{i_1}_!i_1^! \ = \ (j_1{\hat\jmath}_2)^*{i_1}_!i_1^! \ \simeq \ {\hat\jmath}_2^*j_1^*{i_1}_!i_1^!,
$$
and $j_1^*{i_1}_!=0$, since $j_1$ and $i_1$ are inclusions of disjoint sets.
\end{proof}

\section{The Main Theorem}

Let $f:X\rightarrow \C$ be complex analytic, and let $\Adot\in D^b_c(X)$. For any real number $\theta$, let 
$$Z_\theta:= f^{-1}\big(e^{i\theta}\{v\in\C \ | \ \operatorname{Re} v\leq 0\}\big)
$$
and let
$$
L_\theta:= f^{-1}\big(e^{i\theta}\{v\in\C \ | \ \operatorname{Re} v= 0\}\big).
$$
Let $j_\theta:Z_\theta\hookrightarrow X$ and $p:f^{-1}(0)\hookrightarrow X$ denote the inclusions.

Following Exercise VIII.13 of \cite{kashsch} (but reversing the inequality, and using a different shift), we define (or characterize up to natural isomorphism) the shifted vanishing cycles of $\Adot$ along $f$ to be
$$
\phi_f[-1]\Adot  \ := \ p^*R\Gamma_{Z_0}(\Adot) \ \simeq \ p^*{j_0}_!j_0^!\Adot.
$$
In fact, for each $\theta$, we define the shifted vanishing cycles of $\Adot$ along $f$ at the angle $\theta$ to be
$$
\phi^\theta_f[-1]\Adot  \ :=  p^*{j_\theta}_!j_\theta^!\Adot.
$$
There are the well-known natural isomorphisms $\widetilde T^\theta_f:\phi^0_f[-1]\rightarrow \phi^\theta_f[-1]$, induced by rotating $\C$ counterclockwise by an angle $\theta$ around the origin. The natural isomorphism $\widetilde T^{2\pi}_f:\phi_f[-1]\rightarrow \phi_f[-1]$ is the usual monodromy automorphism on the vanishing cycles.

\smallskip

We now prove the main theorem.

\smallskip

\begin{thm} There is a natural isomorphism of functors from $D^b_c(X)$ to $D^b_c(f^{-1}(0))$:
$$
\phi_f[-1]\circ\vdual \ \simeq \ \vdual\circ\phi_f[-1].
$$
\end{thm}
\begin{proof} Let $m$ denote the inclusion of $L_0=L_\pi$ into $X$. Now, apply \lemref{lem2} to $X=Z_0\cup Z_\pi$, and conclude that
$$
m^*{j_0}_!j_0^! \ \simeq \ \hat\jmath_0^!j_\pi^*.
$$
Dualizing, we obtain
\begin{equation}\label{eq3}
\vdual\left(m^*{j_0}_!j_0^!\right) \ \simeq \ \vdual\left(\hat\jmath_0^!j_\pi^*\right) \ \simeq \ \hat\jmath^*_0j^!_\pi\vdual \ \simeq  \ m^*{j_\pi}_!j_\pi^!\vdual,
\end{equation}
where the second isomorphism uses that $\vdual$ ``commutes'' with the standard operations, and the last isomorphism results from using that $m=j_\pi\hat\jmath_0$.

Let $q$ denote the inclusion of $f^{-1}(0)$ into $L_0=L_\pi$, so that the inclusion $p$ equals $mq$. Applying $q^*$ to \eqref{eq3}, we obtain 
$$
q^*\vdual\left(m^*{j_0}_!j_0^!\right) \ \simeq \ q^*m^*{j_\pi}_!j_\pi^!\vdual \ \simeq \ p^*{j_\pi}_!j_\pi^!\vdual  = \phi^\pi_f[-1]\circ\vdual \ \simeq \ \phi_f[-1]\circ\vdual,
$$
where, in the last step, we used the natural isomorphism $\left(\widetilde T^{\pi}_f\right)^{-1}$.

As $\vdual\left(q^!m^*{j_0}_!j_0^!\right)\simeq q^*\vdual\left(m^*{j_0}_!j_0^!\right)$, it remains for us to show that $q^!m^*{j_0}_!j_0^!$ is naturally isomorphic to $q^*m^*{j_0}_!j_0^!\simeq p^*{j_0}_!j_0^!\simeq \phi_f[-1]$.  This will follow from \lemref{lem1}, once we show that, for all $\Adot\in D^b_c(X)$, $f^{-1}(0)\cap \operatorname{supp}(m^*{j_0}_!j_0^!\Adot)$ is an open subset of $\operatorname{supp}(m^*{j_0}_!j_0^!\Adot)$.

Suppose that $x\in f^{-1}(0)\cap \operatorname{supp}(m^*{j_0}_!j_0^1\Adot)$. We need to show that there exists an open neighborhood $\W$ of $x$ in $X$ such that $\W\cap \operatorname{supp}(m^*{j_0}_!j_0^1\Adot)\subseteq f^{-1}(0)$. 

Fix a Whitney stratification of $X$, with respect to which $\Adot$ is constructible. Then, select $\W$ so that all of the stratified critical points of $f$, inside $\W$, are contained in $f^{-1}(0)$. Suppose that there were a point $y\in\W$ such that $y\in f^{-1}(L_0)-f^{-1}(0)$ and the stalk cohomology of $m^*{j_0}_!j_0^!\Adot$ at $y$ is non-zero. Then, by definition, $y$ would be a point in the support $\phi_{f-f(y)}[-1]\Adot$, which, again, is contained in the stratified critical locus of $f$ and, hence, is contained in $f^{-1}(0)$. This contradiction concludes the proof.
\end{proof}

\smallskip

We continue to let $p:f^{-1}(0)\hookrightarrow X$ denote the closed inclusion, and now let $i:X-f^{-1}(0)\hookrightarrow X$ denote the open inclusion. Consider the two fundamental distinguished triangles related to the nearby and vanishing cycles:

$$
p^*[-1]\arrow{\operatorname{comp}}\psi_f[-1]\arrow{\operatorname{can}}\phi_f[-1]\arrow{[1]}
$$
and
$$
\phi_f[-1]\arrow{\operatorname{var}}\psi_f[-1]\rightarrow p^![1]\arrow{[1]}.
$$
The morphisms $\operatorname{comp}$, $\operatorname{can}$, and $\operatorname{var}$ are usually referred to as the {\it comparison map}, {\it canonical map}, and {\it variation map}. As $p^*i_!=0= p^!i_*$ and as $\psi_f[-1]$ depends only on the complex outside of $f^{-1}(0)$, the top triangle, applied to $i_!i^!$ and the bottom triangle applied to $i_*i^*$ yield natural isomorphisms 
$$\alpha:\psi_f[-1]\arrow{\simeq}\phi_f[-1]i_!i^!\hskip .5in\textnormal{and}\hskip .5in  \beta:\phi_f[-1]i_*i^*\arrow{\simeq}\psi_f[-1].
$$

\begin{cor}  There is a natural isomorphism of functors from $D^b_c(X)$ to $D^b_c(f^{-1}(0))$:
$$
\psi_f[-1]\circ\vdual \ \simeq \ \vdual\circ\psi_f[-1].
$$
\end{cor}
\begin{proof}
$$
\psi_f[-1]\circ\vdual \ \simeq \ \phi_f[-1]i_!i^!\circ\vdual \ \simeq \ \vdual\circ \phi_f[-1]i_*i^* \ \simeq \ \vdual\circ \psi_f[-1].
$$
\end{proof}

\newpage
\bibliographystyle{plain}
\bibliography{Masseybib}
\end{document}